\pdfoutput=1
\documentclass[10pt, conference]{IEEEtran}
\usepackage[english]{babel}

\usepackage[margin=54pt]{geometry}

\title{\vspace*{18pt} Distributed PI-Control with Applications to Power~Systems~Frequency~Control}

\author{{Martin Andreasson$^{1}$$^{2}$,  Dimos V. Dimarogonas$^{1}$, Henrik Sandberg$^{1}$ and Karl H. Johansson$^{1}$} \\
\IEEEauthorblockA{\IEEEauthorrefmark{1}ACCESS Linnaeus Center, KTH Royal Institute of Technology, Stockholm, Sweden. }

\thanks{This work was supported in part by the European Commission by the Hycon2 project, the Swedish Research Council (VR) and the Knut and Alice Wallenberg Foundation. We would like to thank the anonymous reviewers for their valuable comments. 
The $2^{\text{nd}}$ author is also affiliated with the Centre for Autonomous Systems at KTH and is supported by the VR 2009-3948 grant. 
$^{2}$ Corresponding author. E-mail: mandreas@kth.se}
}
\IEEEoverridecommandlockouts
 
\usepackage{mathrsfs}
\usepackage{amsthm}
\usepackage{amsfonts}
\usepackage{amssymb}
\usepackage{amsmath}
\usepackage{graphicx}
\usepackage{commath}
\usepackage{tikz}
\usepackage{tkz-graph}
\usetikzlibrary{arrows}
\usepackage{flushend}
\usepackage{psfrag}
\usepackage{pgfplots}
\usepackage{todonotes}
\usepackage{algorithmicx}
\usepackage{algpseudocode}
\usepackage{blockdiagram}
\usetikzlibrary{external}
\newtheorem{theorem}{Theorem}
\newtheorem{corollary}[theorem]{Corollary}
\newtheorem{lemma}[theorem]{Lemma}
\newtheorem{remark}{Remark}

\newtheorem{note}{Note}

\usepackage{algorithm}

\DeclareMathOperator*{\diag}{diag}

\newcommand{\ud}{\,\mathrm{d}}

\newcommand{\beq}{\begin{equation}}
\newcommand{\eeq}{\end{equation}}
\newcommand{\bq}{\begin{eqnarray}}
\newcommand{\eq}{\end{eqnarray}}
\newcommand{\bqn}{\begin{eqnarray*}}
\newcommand{\eqn}{\end{eqnarray*}}
\newcommand{\bee}{\begin{enumerate}}
\newcommand{\eee}{\end{enumerate}}

\newlength\fheight 
\newlength\fwidth 

\begin{document}
\maketitle

\begin{abstract}
This paper considers a distributed PI-controller for networked dynamical systems. Sufficient conditions for when the controller is able to stabilize a general linear system and eliminate static control errors are presented. 
The proposed controller is applied to frequency control of power transmission systems. 
Sufficient stability criteria are derived, and it is shown that the controller parameters can always be chosen so that the frequencies in the closed loop converge to nominal operational frequency. We show that the load sharing property of the generators is maintained, i.e., the input power of the generators is proportional to a controller parameter. The controller is evaluated by simulation on the IEEE 30 bus test network, where its effectiveness is demonstrated. 
\end{abstract}

\section{Introduction}
\label{sec:intro}
Distributed control is the only feasible control strategy for many large-scale systems, when sensing and actuation communication is limited \cite{MorZaf1989}. We will in this paper distinguish between distributed control and decentralized control. In a distributed control architecture, there is no centralized controller with global information, but the controllers can communicate with some of the other controllers and share information. In a decentralized control architecture however, there is no communication between the individual controllers. 
 For systems where constant disturbances or model errors are present, PI-control is a commonly used control strategy, as it will in general eliminate static control errors \cite{aastrom2001future}. For many distributed systems however, decentralized PI-control is known to destabilize the system, as is the case for power transmission systems \cite{machowski2008power}. 

We consider the problem of distributed control of a linear system with the same number of sensors as actuators, and where communication is limited. We show that for a large class of systems, decentralized PI-control is not a feasible control strategy. Instead, we propose a distributed controller, which mimics a decentralized P-controller with a centralized I-controller by distributed averaging. Even though the proportional part of this controller is decentralized, the overall controller is distributed due to the communication needs of the distributed integral part. For a certain class of dynamical systems, the proposed controller is able to eliminate static errors in the output, provided that the closed loop system is output stable, in the sense that all observable modes of the system are stable. 

As mentioned earlier, frequency control of power transmission systems is an important application of distributed PI-control. Traditionally, control with integral action is only carried out by one centralized controller in the power transmission system. However, the increased decentralization of power transmission systems, as well as the independence of micro-grids highlight the need for distributed controllers that do not rely on central coordination. 

A solution to the distributed PI-control problem of power transmission systems has been presented in \cite{Andreasson2012_cdc}. The previously proposed controller however requires phase measurements to be physically implementable. As phase measurements rely on expensive PMUs, it is desirable to study controllers which rely only on local frequency measurements. 
So far, distributed PI-frequency control by distributed averaging has only been considered for a special setting where inverters are used for frequency control in micro-grids \cite{simpson2012droop, simpson2012synchronization}. In these references, stability of the closed-loop power system system was proven, and the controller was shown to preserve the power sharing properties of proportional decentralized frequency controllers with a centralized integrator. A limitation in the analysis is that frequency regulation is assumed to be carried out only by inverters, and not by generators. This also implies that the resulting dynamics of the power transmission system are interconnected first-order differential equations. While it has been shown that the second-order swing equation and the simplified first-order dynamical equation share the same set of equilibra with the same (local) stability properties \cite{dorfler2013synchronization}, the richer second-order dynamics potentially reveal more information about transients. 
In this paper we consider the distributed frequency controller proposed in \cite{simpson2012droop} for a general linear system. We show that the controller can be applied to frequency control of power transmission systems by generator control, where the generator dynamics are modelled by the well-established swing equation \cite{machowski2008power}.
In \cite{Andreasson2013_ecc}, a quadratic generation cost function is introduced, and a distributed algorithm is introduced to minimize the quadratic cost function whilst controlling the frequencies to their nominal value. A solution to the optimization problem was also presented for inverter controlled power transmission systems \cite{bouattour2013distributed}. In this work, we show that the same cost function can be minimized by the proposed distributed PI-controller when carefully selecting controller gains. 

The remaining part of this paper is organized as follows. In Section \ref{sec:model} the model and the problem are introduced. In Section \ref{sec_dec_PI-control} a simple decentralized PI-controller and its limitations are studied. In Section \ref{sec:distributed_PI_control} the distributed PI-controller is introduced and analysed. Section \ref{sec:power_system} applies the previous results to frequency control of power transmission systems by generator control. The paper ends by concluding remarks in Section \ref{sec:discussion}.

\section{Model and problem setup}
\label{sec:model}
Consider a linear system with as many sensors as actuators:
\begin{equation}
\begin{aligned}
\dot{x}(t) &= Ax(t) + Bu(t) + d(t)\\
y(t) &= Cx(t) + \eta(t),
\end{aligned}
\label{eq:linear_system}
\end{equation}
where $x(t) \in \mathbb{R}^n$ is the state, $u(t)\in \mathbb{R}^m$ is the control input, $y(t)\in \mathbb{R}^m$ is the output, $d(t)\in \mathbb{R}^n$ is a disturbance, $\eta(t)\in \mathbb{R}^m$ is measurement noise, and $A\in \mathbb{R}^{n\times n}$, $B\in \mathbb{R}^{n\times m}$, $C\in \mathbb{R}^{m\times n}$. Each sensor is assumed to be coupled with one actuator.
 We refer to each sensor/actuator pair as a \emph{node}. The system is assumed to be physically distributed, making a centralized control architecture infeasible. However, physically neighboring nodes are assumed to be able to communicate directly, and the communication links are modelled by a graph $\mathcal{G}=(\mathcal{V}, \mathcal{E})$, which is assumed to be connected. 
 One important control objective is for the output $y(t)$ to converge to a reference value $r(t)$. We introduce the output error $e(t)=r(t)-y(t)$, and the steady state output error $e_0=\lim_{t\rightarrow \infty}e(t)$. The main control objective can now be stated as $\norm{e_0}=0$. 

\section{Decentralized PI-control}
\label{sec_dec_PI-control}
A simple approach to the control problem detailed in Section \ref{sec:model} is to use a P-controller at each node, i.e.,
\begin{align}
u_i(t) &= K^P_i (r_i(t)-y_i(t)), \label{eq:decentralized_P-control}
\end{align}
where $u_i$ is the $i$'th component of $u$. 
One major drawback with the P-controller however is that $ \norm{e_0}\ne 0$ in general, making it unsuitable when the elimination of static error is essential. A simple and intuitive solution to this problem, is to simply add an integral term to the controller \eqref{eq:decentralized_P-control}:
\begin{align}
u_i(t) &= K^P_i (r_i(t)-y_i(t)) + K^I_i \int_0^{t}  (r_i(\tau)-y_i(\tau)) \ud \tau . \label{eq:decentralized_PI-control}
\end{align}
Unfortunately, this decentralized approach often fails to work in practice for interconnected systems. Define $K^P=\diag([K^P_1, \dots, K^P_m])$ and $K^I=\diag([K^I_1, \dots, K^I_m])$. The following negative result shows that a decentralized PI-controller is infeasible for a certain class of systems. 
\begin{theorem}
\label{th:existence_eq}
 The system \eqref{eq:linear_system} with $u$ given by \eqref{eq:decentralized_PI-control} satisfies $\norm{e_0}=0$ for any constant disturbance $d(t)=d \;\forall t$ and any constant measurement noise $\eta(t) = \eta \;\forall t$ 
only if the matrix
\begin{align}
\Xi&=\begin{bmatrix}
A & BK^I \\
C & 0_{m\times m}
\end{bmatrix}
\label{eq:necessary_cond_1}
\end{align}
has full rank. 
\end{theorem}
\begin{proof}
Introducing $m$ auxiliary integral state states $z$, the dynamics \eqref{eq:linear_system} with the controller \eqref{eq:decentralized_PI-control} can be written as
\begin{align}
\begin{aligned}
\begin{bmatrix}
\dot{x}(t) \\ \dot{z}(t)
\end{bmatrix}
&=\begin{bmatrix}
A-BK^PC & BK^I \\
-C & 0_{m\times m}
\end{bmatrix}\begin{bmatrix}
{x}(t) \\ {z}(t)
\end{bmatrix} \\
&+ 
\begin{bmatrix}
I_n \\ 0_{m\times n}
\end{bmatrix} d +
\begin{bmatrix}
BK^P \\ I_m
\end{bmatrix} (r-\eta).
\end{aligned}
\label{eq:cl_dynamics_1}
\end{align}
Setting $\dot{z}=0_{n\times 1}$ yields
\begin{align}
Cx = - \eta +r. \label{eq:steady_state_1}
\end{align}
Substituting \eqref{eq:steady_state_1} in \eqref{eq:cl_dynamics_1} and setting $\dot{x}=0_{n\times 1}$ yields
\begin{align}
0&= Ax + BK^Iz+d \label{eq:steady_state_2}
\end{align}
Clearly \eqref{eq:steady_state_1} and \eqref{eq:steady_state_2} have a solution for any $r, \eta, d$ if and only if $\Xi\xi=\zeta$ has a solution for any $\zeta$. Thus \eqref{eq:cl_dynamics_1} has an equilibrium only if $\Xi$ has full rank.
\end{proof}
\section{Distributed PI-control by averaging}
\label{sec:distributed_PI_control}
In this section we explore a distributed PI-controller. Recall that the control system is equipped a communication layer, which is represented by the graph $\mathcal{G}$. Let $\mathcal{N}_i$ denote the neighbor set of node $i$. We assume that only neighbors can communicate directly with each other. 
The proposed controller takes the form:
\begin{align}
\begin{aligned}
\dot{z}_i(t) & = (r_i(t)-y_i(t)) - \gamma \sum_{j\in \mathcal{N}_i} c_{ij}(z_i(t)-z_j(t)) \\
u_i(t) &= K^P_i (r_i(t)-y_i(t)) + K^I_i z_i(t),
\end{aligned} \label{eq:distributed_lag-control}
\end{align}
where $K^P_i>0, K^I_i>0, \gamma>0, i=1, \dots, m$, $c_{ij}=c_{ji}>0, i=1,\dots , m, j\in \mathcal{N}_i$ are controller parameters. Define the weighted Laplacian matrix of the undirected communication graph by its entries:
\begin{align*}
\mathcal{L}_{c,ii} &= \sum_{j\in \mathcal{N}_i} c_{ij} \\
\mathcal{L}_{c,ij} &= \left\{ \begin{array}{ll}
- c_{ij} & \text{ if } j \in \mathcal{N}_i \\
0 & \text{ otherwise.} 
\end{array} \right.
\end{align*}
We will show that this controller can be applied to a wider class of systems than the decentralized PI-controller \eqref{eq:decentralized_PI-control}.
Provided that stability can be proven,
the steady-state output error can be shown to vanish under certain conditions. 
\begin{theorem}
\label{th:steady_state_distributed_lag}
Assume that the system \eqref{eq:linear_system} with the controller \eqref{eq:distributed_lag-control} is output stable for given $K^P$, $K^I$, $\mathcal{L}_c$ and $\gamma$, i.e., that all observable modes of $(F, [C,0_{}m\times m])$ are stable, where
\begin{align*}
F&={\begin{bmatrix}
A-BK^PC & BK^I \\
-C & -\gamma \mathcal{L}_c
\end{bmatrix}}.
\end{align*}
Assume furthermore that $\eta(t)=0$ and $d(t)=d$. 
If there exists $k \in \mathbb{R}$ and an $x \in \mathbb{R}^n$ such that $Ax-kBK^I1_{m\times 1} + {d}$ is an unobservable mode of $(A,C)$, and $Cx=r$, then the steady state error satisfies $ e_0 = 0$.
\end{theorem}
\begin{note}
The condition that $Ax-kBK^I1_{m\times 1} + {d}$ is an unobservable mode of $(A,C)$ assures that there is a common integral state such that the output error vanishes.
\end{note}

\begin{proof}
The dynamics of \eqref{eq:linear_system} with the controller \eqref{eq:distributed_lag-control} can be written as:
\begin{align}
\begin{aligned}
\begin{bmatrix}
\dot{x}(t) \\ \dot{z}(t)
\end{bmatrix}
&=\underbrace{\begin{bmatrix}
A-BK^PC & BK^I \\
-C & -\gamma \mathcal{L}_c
\end{bmatrix}}_{\triangleq F}
\begin{bmatrix}
{x}(t) \\ {z}(t)
\end{bmatrix} \\
&+ 
\begin{bmatrix}
I_n \\ 0_{m\times n}
\end{bmatrix} d +
\begin{bmatrix}
BK^P \\ I_m
\end{bmatrix} (r-\eta).
\end{aligned}
\label{eq:cl_dynamics_2}
\end{align} 
 Since the closed loop system is assumed to be output stable, letting $\dot{y}=0_{n\times 1}$ gives: 
\begin{align}
\begin{aligned}
&C{\begin{bmatrix}
A-BK^PC & BK^I 
\end{bmatrix}}
\begin{bmatrix}
{x}(t) \\ {z}(t)
\end{bmatrix} 
= - 
C d -
CBK^P  r.
\end{aligned}
\label{eq:equilibrium_proof}
\end{align}  
Assuming that $z=k1_{m\times 1}$ and $Cx= r, Ax=-BK^Iz$ implies that all observable modes of $(A,C)$ are zero, and the output satisfies $y=r$, which implies $e_0=0$. Since by assumption all observable modes are stable, the closed loop system converges to this equilibrium. 
\end{proof}
We will show later that for the application of power systems that the distributed controller \eqref{eq:distributed_lag-control} does indeed stabilize the power system, even though the decentralized controller \eqref{eq:decentralized_PI-control} cannot stabilize the power system. 
The case when $\eta\ne 0$ is also treated separately for the application of the proposed controller to electrical power transmission systems, since general error bounds are hard to obtain. 

\section{Power transmission systems}
\label{sec:power_system}
\subsection{Introduction}
Consider an electrical power transmission system of generators interconnected by power transmission lines. For power transmission systems with purely inductive lines and where the voltages are assumed to be constant, the swing equation can be employed to model the dynamics of the system \cite{machowski2008power}. The swing equation is linearized around the equilibrium where $\delta = 0_{n\times 1}$, and one obtains:
\begin{eqnarray}
\label{eq:swing_vector}
\left[\begin{matrix}
\dot{\delta} \\ \dot{\omega}
\end{matrix}\right] = \underbrace{\left[\begin{matrix}
0_{n\times n} & I_{n} \\
-M \mathcal{L}_k & - M D
\end{matrix}\right]}_{A} 
\left[\begin{matrix}
{\delta} \\ {\omega}
\end{matrix}\right] +
\underbrace{\left[\begin{matrix}
0_{n\times n} \\ M 
\end{matrix}\right]}_{B} u +
\underbrace{\left[\begin{matrix}
0_{n\times 1} \\ M p^m
\end{matrix}\right]}_{d}
\end{eqnarray}
where $\delta= [\delta_1, \dots, \delta_n]^T$ and $\omega= [\omega_1, \dots, \omega_n]^T$ are the phase angles and frequencies of the generators, respectively. $M=\diag(\frac{1}{m_1}, \hdots , \frac{1}{m_n})$ where $m_i$ is the inertia of bus $i$. $D=\diag(d_1, \hdots, d_n)$ are the damping coefficients, $p^m=[p^m_1,\hdots, p^m_n]^T$ are the electrical power loads and $u=[u_i,\hdots, u_n]^T$ are the mechanical input. $\mathcal{L}_k$ is the weighted Laplacian of the power system, with edge weights $k_{ij}$, where $k_{ij} = |V_i||V_j|b_{ij}$, where $|V_i|$ is the absolute value of the voltage of bus $i$, and $b_{ij}$ is the susceptance of the power transmission line $(i,j)$. The control objective considered in this application is frequency control. After a disturbance, which is here an increased or decreased load, the frequencies $\omega_i$ should converge to a nominal reference frequency $\omega^{\text{ref}}$. By defining the output of the power transmission system as $y=\omega$ and letting $r=\omega^{\text{ref}} 1_{n\times 1}$, the control objective can be stated as $\norm{e_0}=0$, where $e_0=\lim_{t\rightarrow \infty} \omega^{\text{ref}} 1_{n\times 1}-\omega$.

\subsection{Decentralized PI-control}
Assuming that each bus $i$ can measure only its own frequency $\omega_i$, we have
\begin{align}
C=\begin{bmatrix}
0_{n\times n} & I_n
\end{bmatrix}.
\label{eq:output_C}
\end{align}
By Theorem \ref{th:existence_eq}, a stabilizing decentralized PI-controller can exist only if 
\begin{align*}
\Xi&=\begin{bmatrix}
0_{n\times n} & I_{n} & 0_{n\times n} \\
-M \mathcal{L}_k & - M D & MK^I \\
0_{n\times n} & I_n & 0_{n\times n}
\end{bmatrix}
\end{align*}
is full rank. It is however clear from the above equation that the first $n$ rows are linearly dependent of the last $n$ rows in general. Hence there exists no stabilizing decentralized PI-controller for the power system \eqref{eq:swing_vector}. This is verified by a simulation on the IEEE 30 bus test network \cite{IEEE30}.
The line admittances were extracted from  \cite{IEEE30} and the voltages were assumed to be 132 kV for all buses. The values of $M$ and $D$ were assumed to be given by $m_i = 10^5\; \text{kg}\,\text{m}^2$ and $d_i = 1 \; s^{-1} \; \forall i \in \mathcal{V}$. 
The controller gains were given by $K^P = 0.8 I_n$ and $K^I=0.04I_n$ respectively. The reference frequency $\omega^{\text{ref}}$ was assumed to be  $50$ Hz. As seen in Figure~\ref{fig:power_system_dec_diverging}, the frequencies diverge. 

\setlength\fheight{2.5cm} 
	\setlength\fwidth{6.8cm}
\begin{figure}
	\centering
$
\begin{array}{c}
	\input{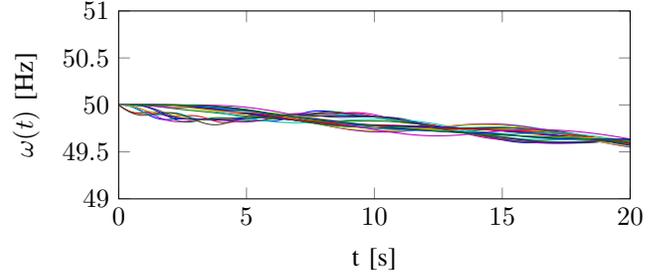}
\end{array}
$
\caption{Bus frequencies with decentralized PI-control and measurement errors. Since the integral states do not converge, the frequencies diverge.}
\label{fig:power_system_dec_diverging}
\end{figure}

\subsection{Distributed PI-control by average consensus}
We show that the controller \eqref{eq:distributed_lag-control} can  be used in control of power transmission systems, where the power flows are governed by the swing-equation \eqref{eq:swing_vector}. We show that the controller achieves asymptotic frequency regulation, while preserving the property of proportional power sharing between the generators. Proportional power sharing based on the ratings of the generators is indeed an important property for generator networks \cite{barsali2002control, hernandez2005fuel} While the controller \eqref{eq:distributed_lag-control} has been applied to frequency control of micro-grids controlled by inverters in \cite{simpson2012droop} and \cite{simpson2012synchronization}, the analysis here is inherently different, since the swing equation is of second-order, as opposed to the first-order models studied in the references. 
The following result establishes the stability of the power system controlled by the distributed PI-controller.

\begin{lemma}
\label{lemma:stability_of_swing}
Assume that the power transmission system \eqref{eq:swing_vector} is controlled by \eqref{eq:distributed_lag-control}, with the reference value given by $r=\omega^{\text{ref}} 1_{n\times 1}$. Assume that $x^T\mathcal{L}_k \mathcal{L}_c x \ge 0$ for all $x\in \mathbb{R}^n$. Then for any $M>0, D>0, \mathcal{L}_k, \mathcal{L}_c, p^m$, and any $K^P>0, K^I>0$ there exists $\bar{\gamma}>0$ such that for all $0<\gamma <\bar{\gamma}$, the closed-loop system is output stable with respect to the output $y=\omega$.
\end{lemma}

\begin{proof}
The power transmission system \eqref{eq:swing_vector} controlled by \eqref{eq:distributed_lag-control} is described by
\begin{align}
\begin{aligned}
\label{eq:cl_power_proof_distributed_PI}
\begin{bmatrix}
\dot{\delta} \\
\dot{\omega} \\
\dot{z}
\end{bmatrix} 
&=
\underbrace{
\begin{bmatrix}
0_{n \times n} & -I_n & 0_{n \times n} \\
-M \mathcal{L}_k & -M(D+K^P) & MK^I \\ 
0_{n \times n} & -I_n & -\gamma \mathcal{L}_c
\end{bmatrix}}_{\triangleq E}
\begin{bmatrix}
\delta \\ \omega \\ z
\end{bmatrix} \\
&+ \begin{bmatrix}
0_{n \times 1} \\
Mp^m \\
0_{n \times 1}
\end{bmatrix}
+
\begin{bmatrix}
0_{n \times n} \\
MK^P \\
I_{n}
\end{bmatrix}
(r-\eta).
\end{aligned}
\end{align}
The output is given by
\begin{align}
y=\underbrace{\begin{bmatrix}
0_{n\times n} & I_n & 0_{n\times n}
\end{bmatrix} }_{\triangleq C}
\begin{bmatrix}
\delta \\ \omega \\ z
\end{bmatrix}
\end{align}
The stability of \eqref{eq:cl_power_proof_distributed_PI} is determined by the eigenvalues of $E$. Consider the characteristic equation of $E$:
\begin{align}
\begin{aligned}
0 &= \det(sI_{3n} - E) \\
 &= \left| \begin{matrix}
sI_n & -I_n & 0_{n \times n} \\
M \mathcal{L}_k & M(D+K^P) + sI_n & -MK^I \\ 
0_{n \times n} & I_n &  sI_n + \gamma \mathcal{L}_c
\end{matrix} \right| \\
&= \frac{1}{\det{(sI_n+\gamma\mathcal{L}_c)}} \cdot \\
& \left| \begin{matrix}
sI_n & -sI_n-\gamma\mathcal{L}_c & 0_{n \times n} \\
M \mathcal{L}_k & \begin{array}{c}
(M(D+K^P) + sI_n)\cdot \\
(sI_n+\gamma\mathcal{L}_c)
\end{array}  & -MK^I \\ 
0_{n \times n} & sI_n+\gamma\mathcal{L}_c & sI_n + \gamma \mathcal{L}_c
\end{matrix} \right| \\
&= 
 \left| \begin{matrix}
sI_n & -sI_n-\gamma\mathcal{L}_c  \\
M \mathcal{L}_k & \begin{array}{c}
(M(D+K^P) + sI_n)\cdot \\
(sI_n+\gamma\mathcal{L}_c) + MK^I
\end{array} \end{matrix} \right| \\
&= \frac{1}{s^n} 
 \left| \begin{matrix}
sI_n & -sI_n-\gamma\mathcal{L}_c  \\
0_{n\times n} & \begin{array}{c}
s(M(D+K^P) + sI_n)\cdot \\
(sI_n+\gamma\mathcal{L}_c) + sMK^I + \gamma M\mathcal{L}_k \mathcal{L}_c
\end{array} \end{matrix} \right|.
\end{aligned}
\end{align}
Expanding the determinant yields
\begin{align}
\begin{aligned}
0&= \det\left( s^3 I_n + s^2M(D+K^P) + s^2\gamma \mathcal{L}_c \right. \\
&\left. + s\gamma M(D+K^P)\mathcal{L}_c + sMK^I + sM\mathcal{L}_k + \gamma M \mathcal{L}_k \mathcal{L}_c \right) \\
&= \det(M) \det\left( s^3 M^{-1} + s^2(D+K^P) + s^2\gamma M^{-1} \mathcal{L}_c \right. \\
&\left. + s\gamma (D+K^P)\mathcal{L}_c + sK^I + s\mathcal{L}_k + \gamma  \mathcal{L}_k \mathcal{L}_c \right) \\
&\triangleq \det(M)\det(Q(s))
\end{aligned}
\label{eq:char_eq_manipulation}
\end{align}
Clearly the above characteristic equation has a solution only if $\exists x : x^TQ(s)x=0$. We may without loss of generality assume that $\norm{x}=1$. Hence we consider
\begin{align}
\begin{aligned}
&x^T\left( s^3 M^{-1} + s^2(D+K^P) + s^2\gamma M^{-1} \mathcal{L}_c \right. \\
&\left. + s\gamma (D+K^P)\mathcal{L}_c + sK^I + s\mathcal{L}_k + \gamma  \mathcal{L}_k \mathcal{L}_c \right)x = 0.
\end{aligned}
\label{eq:central_proof_Q}
\end{align}
If \eqref{eq:central_proof_Q} has all its solutions in $\mathbb{C}^-$ for all $\norm{x}=1$, then \eqref{eq:char_eq_manipulation} has all its solutions in $\mathbb{C}^-$. This condition thus becomes that the equation
\begin{align}
\begin{aligned}
&\underbrace{x^T(\gamma  \mathcal{L}_k \mathcal{L}_c)x}_{a_0} + 
s \underbrace{x^T(\gamma (D+K^P)\mathcal{L}_c + K^I + \mathcal{L}_k)x}_{a_1} \\
& + 
s^2 \underbrace{x^T(\gamma M^{-1} \mathcal{L}_c)x}_{a_2} + 
  s^3 \underbrace{x^TM^{-1}x}_{a_3} = 0
  \end{aligned}
  \label{eq:xQx}
\end{align}
has all its solutions in the complex left half plane. We distinguish between the two cases $x^T \mathcal{L}_k \mathcal{L}_c x=0$ and $x^T \mathcal{L}_k \mathcal{L}_c x \ne 0$, since by assumption $x^T\mathcal{L}_k \mathcal{L}_c x \ge 0$. Starting with the former case, equation \eqref{eq:xQx} may be written as $s(a_1+ a_2s +a_3s^2)$, which has one solution $s=0$, and all remaining solutions $s\in \mathbb{C}^-$ if and only if $a_i>0, i=1,2,3$, by the Routh-Hurwitz condition. For the latter case, \eqref{eq:xQx} has all its solutions $s\in \mathbb{C}^-$ if and only if $a_i>0, i=1,2,3,4$ and $a_0a_3<a_1a_2$. Thus, $E$ has at most one zero eigenvalue, and all remaining eigenvalues in the complex left half plane if $a_i>0, i=1,2,3,4$ and $a_0a_3<a_1a_2$. Clearly $a_3\ge \min_i M^{-1}_i = \max_i m_i>0 \; \forall \norm{x}=1$ and $a_0 > 0$ by assumption. The following lower bounds on the remaining coefficients are easily verified:
\begin{align}
a_1 &\ge \gamma \lambda_{{\min}} \left( \frac{1}{2} (D+K^P)\mathcal{L}_c + \frac{1}{2} \mathcal{L}_c (D+K^P) \right) + \min_{i} K^I_i \label{eq:a_1_lower} \\
a_2 &\ge \gamma \lambda_{\min} \left( \frac{1}{2} M^{-1} \mathcal{L}_c + \frac{1}{2} \mathcal{L}_c M^{-1} \right) + \min_i D_i + K^P_i. \label{eq:a_2_lower}
\end{align}
By \eqref{eq:a_1_lower} and \eqref{eq:a_2_lower}, a lower bound on $a_1a_2$ is obtained:
\begin{align}
\begin{aligned}
&a_1a_2 \ge \\
&\left( \gamma \lambda_{{\min}} \left( \frac{1}{2} (D+K^P)\mathcal{L}_c + \frac{1}{2} \mathcal{L}_c (D+K^P) \right) + \min_{i} K^I_i \right) \cdot \\
&\left( \gamma \lambda_{\min} \left( \frac{1}{2} M^{-1} \mathcal{L}_c + \frac{1}{2} \mathcal{L}_c M^{-1} \right) + \min_i D_i + K^P_i \right).
\label{eq_a_1a_2_lower}
\end{aligned}
\end{align}
By similar upper bounds on $a_0$ and $a_3$, the following upper bound on $a_0a_3$ is obtained:
\begin{align}
\begin{aligned}
&a_0a_3 \le \gamma \left( \min_i m_i \right)  \lambda_{\max} \left( \frac{1}{2} \mathcal{L}_k \mathcal{L}_c + \frac{1}{2} \mathcal{L}_c \mathcal{L}_k \right).
\end{aligned}
\label{eq:a_0a_3_upper}
\end{align}
Clearly, by \eqref{eq:a_1_lower} and \eqref{eq:a_2_lower}, $a_1>0$ and $a_2>0$ for $\gamma=0$. Furthermore $a_0a_3 < a_1a_2$ when $\gamma=0$. By continuity of polynomial functions, there exists $\bar{\gamma}$ such that $a_1>0$, $a_2>0$ and $a_0a_3 < a_1a_2$ $\forall \gamma < \bar{\gamma}$. The right eigenvector $v_0$ corresponding to the zero eigenvalue of $E$ is $v_0=[1_{1\times n}, 0_{1\times n}, 0_{1\times n}]^T$. However, since $v_0$ is an unobservable mode of $(A,C)$, and all other eigenvalues have strictly negative real part, \eqref{eq:cl_power_proof_distributed_PI} is output stable with respect to the output $y=\omega$. 
\end{proof}

\begin{corollary}
\label{corr:dist_freq_control}
Assume that the power transmission system \eqref{eq:swing_vector} is controlled by \eqref{eq:distributed_lag-control}, with the reference value given by $r=\omega^{\text{ref}} 1_{n\times 1}$. Assume that $x^T\mathcal{L}_k \mathcal{L}_c x \ge 0$ for all $x\in \mathbb{R}^n$. Let $M>0, D>0, \mathcal{L}_k, \mathcal{L}_c, p^m$, and $K^P>0, K^I>0$ be arbitrary, and let  $\eta=0_{n\times 1}$. Then there exists $\bar{\gamma}>0$ such that for all $0<\gamma <\bar{\gamma}$ it holds that $\lim_{t\rightarrow \infty} \omega(t) = \omega^{\text{ref}}1_{n\times 1}$ and $\lim_{t\rightarrow \infty} u(t) = k K^I 1_{n\times 1}$, where $k\in \mathbb{R}$. If $\eta \ne 0_{n\times 1}$, then $\lim_{t\rightarrow \infty} \omega(t) = \hat{\omega} 1_{n\times 1}$, where $\hat{\omega} = \omega^{\text{ref}} - {1}/{n} 1_{1\times n} \eta$. 
\end{corollary}
\begin{remark}
A sufficient condition for when $x^T\mathcal{L}_k \mathcal{L}_c x \ge 0$ for all $x\in \mathbb{R}^n$ is that $\mathcal{L}_c=k_1\mathcal{L}_k$, $k_1 \in \mathbb{R}^+$ i.e., the topology of the communication network is identical to the topology of the power transmission lines. 
\end{remark}

\begin{proof}
We will invoke Theorem \ref{th:steady_state_distributed_lag} to show that $e_0=0$. 
By Lemma \ref{lemma:stability_of_swing}, there exists $\bar{\gamma}>0$ such that for all $0<\gamma <\bar{\gamma}$, the power transmission system \eqref{eq:swing_vector} controlled by \eqref{eq:distributed_lag-control} is output stable. 
Furthermore $\eta=0_{n}$ and $d(t)=d$. Letting $x=[\delta^T, \omega^T]^T$ and setting $\omega = \omega^{\text{ref}}1_{n\times 1}$, it clearly holds that $Cx=r$. For the power transmission system, it is easy to show that $x^T=[\delta^T,\omega^T]^T=[1_{1\times n}, 0_{1\times n}]$ is an unobservable mode of $A$, since
\begin{align*}
Cx&= \begin{bmatrix}
0_{n\times n} & I_n
\end{bmatrix} 
\begin{bmatrix}
1_{n\times 1}\\ 0_{n\times 1}
\end{bmatrix} =
0_{2n\times 1} \\
Ax &= \left[\begin{matrix}
0_{n\times n} & I_{n} \\
-M \mathcal{L}_k & - M D
\end{matrix}\right] \begin{bmatrix}
1_{n\times 1}\\ 0_{n\times 1}
\end{bmatrix} = 0_{2n\times 1},
\end{align*}
which implies that $\mathcal{O}x=0_{2n\times 1}$, where 
\begin{align*}
\mathcal{O} &= \begin{bmatrix}
C \\
CA \\
\vdots \\
CA^{2n-1}
\end{bmatrix}
\end{align*}
is the observability matrix. It is clear that $A$ has rank $2n-1$, which implies that $\mathcal{O}$ also must have rank $2n-1$. Thus, we need to verify that there exist $x=[\delta^T, \omega^T]^T=[\delta^T, \omega^{\text{ref}}1_{1\times n}]^T$ and $k$ such that $Ax-kBK^I1_{2n\times 1} +d = k_2[1_{1\times n}, 0_{1\times n}]^T$. This condition can be written as
\begin{align}
\begin{aligned}
&\left[\begin{matrix}
0_{n\times n} & I_{n} \\
-M \mathcal{L}_k & - M D
\end{matrix}\right] 
\begin{bmatrix}
\delta \\ \omega^{\text{ref}}1_{n\times 1}
\end{bmatrix} - k \begin{bmatrix}
0_{n\times 1} \\ MK^I1_{n\times 1} 
\end{bmatrix} + 
\begin{bmatrix}
0_{n\times 1} \\ Mp^m
\end{bmatrix} \\
&= k_2\begin{bmatrix}
1_{n\times 1}\\ 0_{n\times 1}
\end{bmatrix}.
\end{aligned}
\label{eq:cond_th_dist_PI_power}
\end{align}
The first $n$ rows of \eqref{eq:cond_th_dist_PI_power} are satisfied if we let $k_2=\omega^{\text{ref}}$. Since $M$ is full rank, the last $n$ rows are equivalent to
\begin{align*}
\mathcal{L}_k \delta  - K^I1_{n\times 1} k = -p^m +(D+K^P)1_{n\times 1} \omega^{\text{ref}},
\end{align*}
which can be written in matrix form as
\begin{align*}
\begin{bmatrix}
\mathcal{L}_k & - K^I1_{n\times 1}
\end{bmatrix}
\begin{bmatrix}
 \delta  \\ k
\end{bmatrix}
 = -p^m +(D+K^P)1_{n\times 1} \omega^{\text{ref}}. 
\end{align*}
The above equation has a solution $[\delta^T,k]^T$ for any $-p^m +(D+K^P)1_{n\times 1}$ if and only if $[\mathcal{L}_k, - K^I1_{n\times 1}]$ has rank $n$. Consider:
\begin{align*}
x'^T
\begin{bmatrix}
\mathcal{L}_k & - K^I1_{n\times 1}
\end{bmatrix} = 0_{1\times (n+1)}.
\end{align*} 
The first $n$ columns of the above equation imply $x'=k_3 1_{n\times 1}$. Inserting this in the last column of the above equation yields $k_3 1_{1\times n}K^I1_{n\times 1}=0$, implying $k_3=0$, since the diagonal elements of $K^I$ are strictly positive. Hence $[\mathcal{L}_k, - K^I1_{n\times 1}]$ has rank $n$, and \eqref{eq:cond_th_dist_PI_power} has a solution, and $Ax-kBK^I1_{m\times 1} + {d}$ is an unobservable mode of $(A,C)$. Thus, by Theorem \ref{th:steady_state_distributed_lag}, $e_0=0$. 

We now consider explicitly the case when $\eta\ne 0$, and also study the control signals $u_i$. 
Consider the coordinate change 
\begin{align*}
\delta
 &= \begin{bmatrix}
\frac{1}{\sqrt{n}} 1_{n\times 1} & S
\end{bmatrix} 
\delta' 
 \quad 
\delta' 
 = \begin{bmatrix}
\frac{1}{\sqrt{n}}1_{1\times n} \\ S^T
\end{bmatrix} 
\delta.
\end{align*}
where $S$ is a matrix such that $ \begin{bmatrix}
\frac{1}{\sqrt{n}}1_{n\times 1} & S 
\end{bmatrix}$ is an orthonormal matrix.
In the new coordinates the system dynamics \eqref{eq:cl_power_proof_distributed_PI} are given by:
\begin{align}
\begin{aligned}
\label{eq:cl_power_proof_distributed_PI_delta'}
\begin{bmatrix}
\dot{\delta'} \\
\dot{\omega} \\
\dot{z}
\end{bmatrix} 
&=
{
\begin{bmatrix}
0_{n \times n} & -\begin{bmatrix}
\frac{1}{\sqrt{n}} 1_{1\times n} \\ S^T
\end{bmatrix} & 0_{n \times n} \\
\begin{bmatrix}
0_{n\times 1} & -M \mathcal{L}_k S
\end{bmatrix}
 & -M(D+K^P) & MK^I \\ 
0_{n \times n} & -I_n & -\gamma \mathcal{L}_c
\end{bmatrix}}
\begin{bmatrix}
\delta' \\ \omega \\ z
\end{bmatrix} \\
&+ \begin{bmatrix}
0_{n \times 1} \\
Mp^m \\
0_{n \times 1}
\end{bmatrix}
+
\begin{bmatrix}
0_{n \times n} \\
MK^P \\
I_{n}
\end{bmatrix}
(r-\eta).
\end{aligned}
\end{align}
The state $\delta'_1$ is clearly unobservable, and dropping this state by defining $\delta'' = [\delta_2, \dots, \delta_n]^T$ yields the following dynamics
\begin{align}
\begin{aligned}
\label{eq:cl_power_proof_distributed_PI_delta''}
\begin{bmatrix}
\dot{\delta''} \\
\dot{\omega} \\
\dot{z}
\end{bmatrix} 
&=
\underbrace{
\begin{bmatrix}
0_{(n-1) \times n} & -S^T & 0_{(n-1) \times n} \\
 -M \mathcal{L}_k S
 & -M(D+K^P) & MK^I \\ 
0_{n \times n} & -I_n & -\gamma \mathcal{L}_c
\end{bmatrix}}_{\triangleq D'}
\begin{bmatrix}
\delta'' \\ \omega \\ z
\end{bmatrix} \\
&+ \begin{bmatrix}
0_{n \times 1} \\
Mp^m \\
0_{n \times 1}
\end{bmatrix}
+
\begin{bmatrix}
0_{n \times n} \\
MK^P \\
I_{n}
\end{bmatrix}
(r-\eta).
\end{aligned}
\end{align}
The matrix $D'$ is easily shown to be Hurwitz by following the steps of the proof of Lemma \ref{lemma:stability_of_swing}. 
Explicitly  computing the equilibrium of \eqref{eq:cl_power_proof_distributed_PI_delta''} yields that the first $n$ rows $S^T\omega = 0_{(n-1)\times 1}$, implying $\omega=\hat{\omega}1_{n\times 1}$. Inserting this in the last $n-1$ rows of \eqref{eq:cl_power_proof_distributed_PI_delta''} yields $ (\omega^{\text{ref}} - \hat{\omega})1_{n\times 1} -\gamma \mathcal{L}_c z = \eta$. Premultiplying with $1_{1\times n}$ yields $ (\omega^{\text{ref}} - \hat{\omega})n  = 1_{1\times n}\eta$, or equivalently $\hat{\omega} = \omega^{\text{ref}} - \frac{1}{n} 1_{1\times n} \eta$. If $\eta =0$, then $\hat{\omega}=\omega^{\text{ref}} $, and furthermore the last $n$ rows of \eqref{eq:cl_power_proof_distributed_PI_delta''} imply $z=k_4 1_{n\times 1}$. Thus, at stationarity $u_i = K^P_i (r_i-y_i) + K^I_i z_i(t) =  k_4 K^I_i$, which concludes the proof. 
\end{proof}
Corollary \ref{corr:dist_freq_control} has several important consequences. Firstly, if the integral gains are chosen uniformly, then at stationarity $u_i=u_j \; \forall i,j \in \mathcal{V}$, i.e., power is shared equally amongst the generators. 
Secondly, the distributed PI-controller can asymptotically minimize the quadratic generation cost $\sum_{i \in \mathcal{V}} \frac 12 C_i u_i^2$ s.t. $\mathcal{L}_k \delta -u = P^m -\omega^{\text{ref}} D 1_{n \times 1}$. This requires the integral gains to be chosen as $K^I=C^{-1}$, where $C=[C_1, \dots, C_n]$. For a proof, please refer to \cite{Andreasson2013_ecc}.

\subsection{Simulations}
\begin{figure}[ht]
\begin{center}
\includegraphics[width=\columnwidth-0mm]{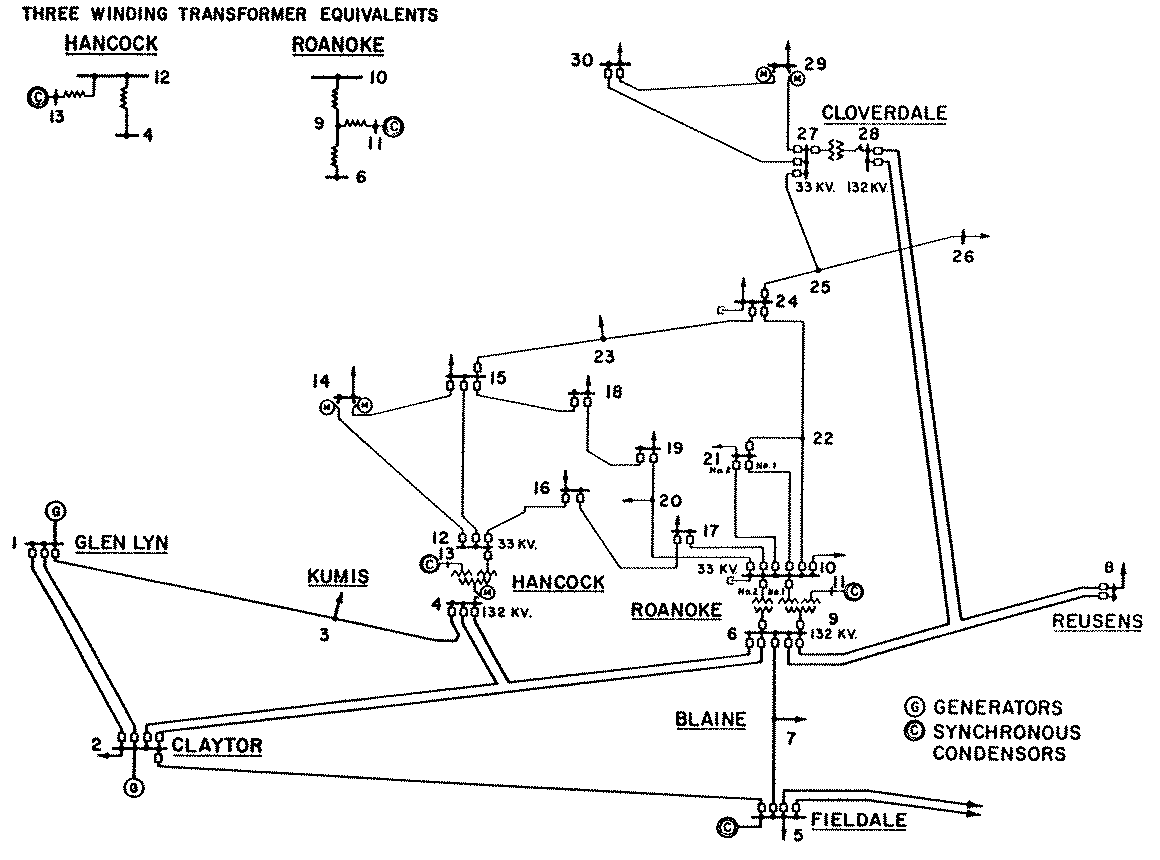}
\caption{The IEEE 30-bus test system, an example of an electrical power system.}
\label{fig:ieee30}
\end{center}
\end{figure}
The power transmission system \eqref{eq:swing_vector} controlled by \eqref{eq:distributed_lag-control} was simulated on the IEEE 30 bus test system, illustrated in Figure \ref{fig:ieee30}. 
The line admittances were extracted from the {IEEE 30} bus test system, and the voltages were assumed to be 132 kV for all buses. The values of $M$ and $D$ were assumed to be given by $m_i = 10^5\; \text{kg}\,\text{m}^2$ and $d_i = 1 \; s^{-1}$, respectively, for all $i \in \mathcal{V}$. The controller gains were given by $K^P=80000I_n$ Ws and $K^I=40000I_n$ W. The communication topology was assumed to be identical with the topology of the power transmission system, i.e., $\mathcal{L}_c=\mathcal{L}_k$. 
The power system is initially in an operational equilibrium, until the power load is increased by a step of $200$ kW in the buses $2,3$ and $7$. This will immediately result in decreased frequencies at the buses where the load is increased as well as in neighboring buses. Subsequently, the frequencies are restored by the distributed PI-controller. 
The step responses of the frequencies are plotted in Figure~\ref{fig:powersystems_sim_1}. The distributed PI-controller quickly regulates the frequencies to the nominal frequency, while the power injections $u_i$ quickly reach an operating point, where all power injections are equal.

	\setlength\fwidth{6.8cm}
\begin{figure}[t]
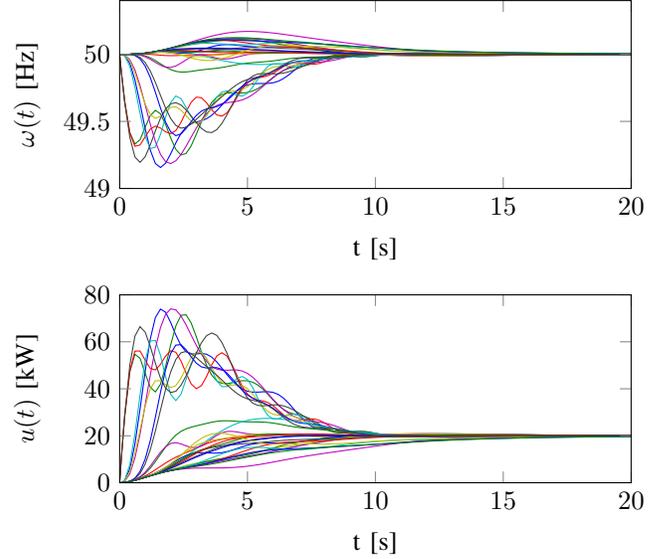

	\centering
$
\begin{array}{c}
	\input{Simulations/Powersystems/omega_dec_noise.tikz} \\
\input{Simulations/Powersystems/u_dec_noise.tikz}
\end{array}
$
\caption{The figures show the bus frequencies and control signals, respectively, of the power system \eqref{eq:swing_vector} controlled by \eqref{eq:distributed_lag-control} under a step load increase.} 
\label{fig:powersystems_sim_1}
\end{figure}

\section{Discussion and Conclusions}
\label{sec:discussion}
In this paper we have considered a distributed PI-controller for networked dynamical systems. Sufficient conditions for when the controller eliminates static control errors were presented. 
The proposed controller was applied to frequency control of power transmission systems by generator control. We showed that the proposed controller regulates the bus frequencies of the power system towards a common reference frequency, while satisfying the power sharing property between the generators. It was shown that there always exist control parameters such that the controlled power transmission system is asymptotically output stable, in the sense that the frequencies converge to the nominal frequency. 

\bibliography{references}
\bibliographystyle{plain}
\end{document}